\documentclass[preprint,11pt]{elsarticle}
\usepackage{amsmath,amscd,amsthm,amssymb,fullpage}
\usepackage{url}
\usepackage{floatflt}

\newtheorem{thm}{Theorem}[section]
\newtheorem{lem}[thm]{Lemma}

\newtheorem{prop}[thm]{Proposition}
\newtheorem{cor}[thm]{Corollary}

\newtheorem{obs}[thm]{Remark}

\newtheorem{example}[thm]{Example}

\newcommand{\forme}[1]{}

\newcommand{\gn}[1]{\langle {#1}\rangle}
\newcommand{\Qd}{/\!\!/}

\newcommand{\Or}{{\mathbf O}^{\mathrm \theta}}

\begin{document}
\begin{frontmatter}

\title{On $p$-schemes of order $p^3$}
\begin{abstract}
Let $(X,S)$ be a $p$-scheme of order $p^3$ and $T$ the thin residue of $S$.
Now we assume that $T$ has valency $p^2$.
It is easy to see that one of the following holds:
(i) $|T|=p^2$ and $T\simeq C_{p^2}$;
(ii) $|T|=p^2$ and $T\simeq C_p\times C_p$;
(iii) $|T|<p^2$.

It is known that $(X,S)$ is Schurian if (i) holds.
If (ii) holds, we will show that $(X,S)$ induces a partial linear space on $X/T$. Moreover, the character degrees of $(X,S)$ coincide with
the sizes of the lines of the partial linear space.
Under the assumption (iii) we will show a construction of non-Schurian $p$-schemes which are algebraically
isomorphic to a Schurian $p$-scheme of order $p^3$.
\end{abstract}

\author{Jung Rae Cho \fnref{fn1}}
\ead{jungcho@pusan.ac.kr}
\address{Department of Mathematics, College of Natural Sciences, Pusan National University,\\ Busan 609--735, Republic of Korea.}
\fntext[fn1]{This work was supported for two years by Pusan National University Research Grant.}

\author{Mitsugu Hirasaka \fnref{fn2}}
\ead{hirasaka@pusan.ac.kr}
\address{Department of Mathematics, College of Natural Sciences, Pusan National University,\\ Busan 609--735, Republic of Korea.}
\fntext[fn2]{This research was supported by Basic Research Program through the National Research Foundation of Korea(NRF) funded by the Ministry of Education, Science and Technology(grant number 2010-0009675).}

\author{Kijung Kim\corref{cor}\fnref{fn3}}
\ead{knukkj@pusan.ac.kr}
\address{Department of Mathematics, College of Natural Sciences, Pusan National University,\\ Busan 609--735, Republic of Korea.}
\fntext[fn3]{This work was supported by BK21 Dynamic Math Center, Department of Mathematics at Pusan National University.}
\cortext[cor]{Corresponding Author}



\end{frontmatter}

\section{Introduction}\label{sec:intro}
An association scheme is a combinatorial object which is defined by some algebraic properties
derived from a transitive permutation group (see Section~2 for definitions).
Actually, any transitive permutation group $G$ on a finite set $\Omega$
induces an associations scheme $(\Omega,\mathcal{R}_G)$, where $\mathcal{R}_G$ is the set of
orbits of the induced action of $G$ on $\Omega\times \Omega$ (see \cite[Example 2.1]{bannai}).
On the other hand, we can find many association schemes which are not induced by groups
(see \cite{hanakimi}).

We say that an association scheme $(X,S)$ is \textit{Schurian} if
$S=\mathcal{R}_G$ for a transitive permutation group $G$ of $X$.
In order to determine whether a given association scheme $(X,S)$ is Schurian or not,
it is quite useful to consider its thin residue, i.e., the smallest subset $T$ of $S$
such that $\bigcup_{t\in T}t$ is an equivalence relation on $X$ and the factor scheme of $(X,S)$ over $T$
is induced by a regular permutation group (see Section~2 for definitions).
For example, if $T$ is a group under the relational product and the set of normal subgroups of $T$ is totally ordered with respect to the set-theoretic inclusion, then $(X,S)$ is Schurian (see \cite{hiszies}).
In this paper we consider $p$-schemes which generalize $(\Omega,\mathcal{R}_G)$ for $p$-groups $G$
and characterize them by their thin residues.

Let $(X,S)$ be a $p$-scheme, where $p$ is a prime.
Then $|X|$ is a power of $p$.
It is known that $(X,S)$ is Schurian if $|X|\in \{1,p,p^2\}$, and
there are exactly three $p$-schemes of order $p^2$.
Recall that the number of isomorphism classes of $p$-groups of order $p^3$
is five. In contrast to this situation the number of isomorphism classes of
$p$-schemes of order $p^3$ depends on the choice of $p$
(see \cite{hanakimi}), and there is a method to construct non-Schurian $p$-schemes of order $p^3$
for each $p>5$ (see \cite{banghi}).
By the discussion in the second paragraph, $(X,S)$ is Schurian if the thin residue $T$ of $S$
is a cyclic group under the relational product.
Thus, the following cases are left when we enumerate $p$-schemes of order $p^3$ with thin residue $T$ of valency $p^2$:
\begin{enumerate}
\item $T$ is the elementary abelian group of order $p^2$;
\item $T$ is not a group.
\end{enumerate}

In the first case, we will obtain a partial linear space $(\mathcal{P},\mathcal{L})$ (see Section~2 for details).
According to \cite{banghi} it is expected that quite many non-Schurain $p$-schemes of order $p^3$ can be
obtained. However, it seems far from reach to classify all $p$-schemes of order $p^3$.
We will consider $\{|L|\mid L\in \mathcal{L}\}$ and show that
it coincides with $\{\chi(I)\mid \chi \in \mathrm{Irr}(S)\}$,
where $I$ is the identity matrix and $\mathrm{Irr}(S)$ is the set of irreducible characters of the adjacency algebra of $(X,S)$ (see Section~2 for definitions).
This result is obtained as Theorem \ref{thm:main3} of this paper.

In the second case, it is known that there are three pairs of algebraically isomorphic $3$-schemes of order 27,
each consisting of one Schurian scheme and one non-Schurian scheme.
In this paper we will generalize this fact to obtain the following:

\begin{thm}\label{thm:main1}
Let $(X,S)$ be an association scheme, $T$ a closed subset of $S$, and $X_1, X_2, \dots, X_m$ the cosets of $T$ in $X$.
Suppose that $\iota_{_1}, \iota_{_2}, \dots, \iota_{_m} \in \mathrm{Aut}(T)$ extend to $\tilde{\iota}_{_1}, \tilde{\iota}_{_2}, \dots, \tilde{\iota}_{_m} \in \mathrm{Aut}(S)$ by $s^{\tilde{\iota}_{_j}} = s$ if $s \in S \setminus T$ and $s^{\tilde{\iota}_{_j}}= s^{\iota_{_j}}$ if $s \in T$ for all $j=1, \dots, m$.
Let $S' = \{s' \mid s \in S \}$, where $s'=s$ if $s \in S \setminus T$ and $s' = \bigcup_{j=1}^m (s^{\iota_{_j}} \cap (X_j \times X_j))$ if $s \in T$.
Then $(X,S')$ is an association scheme which is algebraically isomorphic to $(X,S)$.
\end{thm}

\begin{thm}\label{thm:main2}
Let $H$ and $T$ be closed subsets of $S$ such that $H \subseteq T$ and $T$ is strongly normal in $S$.
Assume that the following conditions are satisfied:
\begin{enumerate}
\item $sH=sT$ for each $s \in S \setminus T$;
\item $tH=t$ for each $t \in T \setminus H$.
\end{enumerate}
Then, for each $\iota \in \mathrm{Aut}(T)$ fixing each element of $H$, we have $\tilde{\iota} \in \mathrm{Aut}(S)$, where $s^{\tilde{\iota}} = s$ if $s \in S \setminus T$ and $s^{\tilde{\iota}}= s^{\iota}$ if $s \in T$.
\end{thm}

As a corollary of these theorems we can construct quite a many $p$-schemes of order $p^3$
by taking arbitrary sets of automorphisms of the cyclic group of order $p$.

This paper is organized as follows.
In Section \ref{sec:pre}, we review notations and terminology on association schemes and partial linear spaces.
In Section \ref{sec:main1}, we show that the irreducible character degrees coincide with
the sizes of the lines of the partial linear space.
In Section \ref{sec:main2}, we give proofs of Theorem \ref{thm:main1} and \ref{thm:main2}, and
apply these results to a class of $p$-schemes obtaining association schemes with the same intersection numbers.

\section{Preliminaries}\label{sec:pre}
In this section, we review some notations and known facts about association schemes and partial linear spaces.
Throughout this paper, we use the notations given in \cite{hanaki1,zies}.

\subsection{Association schemes}\label{sec:as}
Let $X$ be a non-empty finite set.
Let $S$ denote a partition of $X \times X$. Then we say that $(X, S)$ is
an \textit{association scheme} (or shortly  \textit{scheme}) if it satisfies the following conditions:

\begin{enumerate}
\item $1_{_X} := \{ (x, x) \mid  x \in X \} \in S$;
\item For each $s \in S$, $s^* := \{(x, y) \mid (y, x) \in s \} \in S$;
\item For all $ s, t, u \in S$ and $x, y \in X$, $c_{st}^u :=| xs \cap yt^* |$ is constant whenever $(x, y) \in u$,
where $xs:=\{ y \in X \mid (x,y) \in s \}$.
\end{enumerate}

For each $ s \in S $, we abbreviate $ c_{ss^*}^{1_{_X}} $ as $n_s$, which is called the \textit{valency} of $s$.
For a subset $U$ of $S$, put $n_{_U} = \sum_{u \in U} n_u$.
We call $n_{_S}$ the \textit{order} of $(X, S)$.
In particular, $(X,S)$ is called a \textit{p-scheme} if $|X| \prod_{s \in S} n_s$ is a power of $p$, where $p$ is a prime.

Let $P$ and $Q$ be non-empty subsets of $S$. We define $PQ$ to be the set of all elements $s \in S$ such that there exist element
$p \in P$ and $q \in Q$ with $c_{pq}^s \neq 0$. The set $PQ$ is called the \textit{complex product} of $P$ and $Q$.
If one of factors in a complex product consists of a single element $s$, then one usually writes  $s$ for $\{ s \}$.

A non-empty subset $T$ of $S$ is called \textit{closed} if $TT \subseteq T$.
Note that a subset $T$ of $S$ is closed if and only if $\bigcup_{t \in T}t$ is an equivalence relation on $X$.
A closed subset $T$ is called \textit{thin} if all elements of $T$ have valency 1.
The set $\{ s \mid n_s=1 \}$ is called the \textit{thin radical} of $S$ and denoted by $\mathbf{O}_\mathrm{\theta} (S)$.
Note that $T$ is thin if and only if $T$ is a group with respect to the relational product.

For a closed subset $T$ of $S$, $n_{_S} / n_{_T}$ is called the \textit{index} of $T$ in $S$.
A closed subset $T$ of $S$ is called \textit{strongly normal} in $S$, denoted by $T \lhd^\sharp S$, if $s^* T s \subseteq T$ for every $s \in S$.
We put $\Or (S) := \bigcap_{T \lhd^\sharp S} T $ and call it the \textit{thin residue} of $S$.
Note that $\Or (S)$ is the intersection of all closed subsets of $S$ which contain $\bigcup_{s \in S} s^*s$ (see \cite[Theorem 2.3.1]{zies}).

For each closed subset $T$ of $S$, we define $X/T:= \{ xT \mid x \in X\}$, called the set of \textit{cosets} of $T$ in $X$, and $S/\!/T:=\{ s^T \mid s \in S \}$,
where $xT:= \bigcup_{t \in T} xt$ and $s^T:= \{ (xT,yT) \mid y \in xTsT \}$. Then $(X/T, S/\!/T)$(or shortly $S/\!/T$) is a scheme
called the \textit{quotient}(or \textit{factor}) scheme of $(X,S)$ over $T$.
Note that $T \lhd^\sharp S$ if and only if $S/\!/T$ is a group (see \cite[Theorem 2.2.3]{zies}).

Let $(W,F)$ and $(Y,H)$ be association schemes.
For each $f \in F$ we define
\[ \overline{f}:= \{ ((w_1,y),(w_2,y)) \mid y \in Y, (w_1,w_2) \in f  \} .\]
For each $h \in H \setminus \{1_{_Y}\}$ we define
\[ \overline{h}:=\{ ((w_1,y_1),(w_2,y_2)) \mid w_1, w_2 \in W, (y_1,y_2) \in h \}. \]
Let $F \wr H := \{ \overline{f} \mid f \in F \} \cup \{\overline{h} \mid h \in H \setminus \{1_{_Y}\} \}$.
Then $(W \times Y, F \wr H)$ is an association scheme called the \textit{wreath product} of $(W,F)$ and $(Y,H)$.

Let $(X,S)$ and  $(X_1, S_1)$ be association schemes. A bijective map $\phi : X \cup S \rightarrow X_1 \cup S_1$ is called an \textit{isomorphism} if it satisfies the following conditions:
\begin{enumerate}
\item $X^\phi \subseteq X_1$ and $S^\phi \subseteq S_1$;
\item For all $x, y \in X$ and $s \in S$ with $(x,y) \in s$, $(x^\phi, y^\phi) \in s^\phi$.
\end{enumerate}
An isomorphism $\phi : X \cup S \rightarrow X \cup S$ is called an \textit{automorphism} of $(X,S)$ if $s^{\phi}=s$ for all $s \in S$.
We denote by  the $\mathrm{Aut}(X,S)$ the automorphism group of $(X,S)$.

On the other hand, we say that $(X,S)$ and  $(X_1, S_1)$ \textit{are algebraically isomorphic} or \textit{have the same intersection numbers} if there exists a bijection $\iota : S \rightarrow S_1$ such that $c_{rs}^{t} = c_{r^\iota s^\iota} ^{t^\iota}$ for all $r, s, t \in S$.
For a closed subset $T$ of $S$, we denote by $\mathrm{Aut}(T)$ the set of permutations of $T$ which preserve the intersection numbers, namely
\[ \mathrm{Aut}(T) = \{ \iota \in \mathrm{Sym}(T) \mid c_{rs}^{t} = c_{r^\iota s^\iota}^{t^\iota} ~\text{for all}~ r, s, t \in T \}.\]

For each $s \in S$, we denote the \textit{adjacency matrix} of $s$ by $\sigma_s$.
Namely $\sigma_s$ is a matrix whose rows and columns are indexed by the elements of $X$ and $(\sigma_s)_{xy} = 1$ if $(x,y) \in s$ and $(\sigma_s)_{xy} = 0$ otherwise.
We denote by $\mathbb{C}S$ the subspace of $\mathrm{Mat}_{_{|X|}}(\mathbb{C})$ spanned by $\{ \sigma_s \mid s \in S \}$.
 It follows from the definition of a scheme that $\mathbb{C}S = \bigoplus_{s \in S} \mathbb{C}\sigma_s$ is a subalgebra of $\mathrm{Mat}_{_{|X|}}(\mathbb{C})$, and it is called the \textit{adjacency algebra} of $(X,S)$.
For a subset $U$ of $S$,
we put $\mathbb{C}U = \bigoplus_{u \in U} \mathbb{C}\sigma_u$ as a subset of $\mathbb{C}S$.
It is well-known that $\mathbb{C}S$ is a semisimple algebra (see \cite{zies} Theorem 4.1.3).
The set of irreducible characters of $S$ is denoted by $\mathrm{Irr}(S)$. The map $\sigma_s \mapsto n_s$ is called the \textit{principal} character of $S$
and denoted by $1_{_S}$.
Now $\mathbb{C}X=\bigoplus_{x \in X}\mathbb{C}x$ is a $\mathbb{C}S$-module by $x\sigma_s = \sum_{y \in xs}y$ (see \cite{zies} p.97).
Let $\tau_{_S}$ be the character afforded by $\mathbb{C}X$. We call $\tau_{_S}$ the \textit{standard} character of $S$.
For each irreducible character of $S$, there exists a non-negative integer $m_\chi$ such that $\tau_{_S} = \sum_{\chi \in \mathrm{Irr}(S)} m_\chi \chi$. The number $m_\chi$ is called the \textit{multiplicity} of $\chi$.
For $\chi \in \mathrm{Irr}(S)$, the central primitive idempotent of $\mathbb{C}S$ corresponding to $\chi$ is denoted by $e_\chi$.
Note that $\mathrm{Irr}(S/\!/\Or (S))$ is embedded in $\mathrm{Irr}(S)$ (see \cite{factor}).

\begin{thm}[see \cite{muzy}]\label{thm:residue}
Let $\chi \in \mathrm{Irr}(S)$. Then $\chi \in \mathrm{Irr}(S/\!/\Or (S))$ if and only if $m_\chi = \chi(\sigma_{1_{_X}})$.
\end{thm}

Let $\mathfrak{X}$ be a representation of $\mathbb{C}S$ affording the character $\chi$.
We let $\mathfrak{X}(\sigma_s)=(x_{ij}(\sigma_s))$. Let $\rho$ be a field automorphism of $\mathbb{C}$.
Then $\mathfrak{X}^\rho : \sigma_s \mapsto (x_{ij}(\sigma_s)^\rho)$ is also a representation of $\mathbb{C}S$,
and $\mathfrak{X}^\rho$ is irreducible if and only if so is $\mathfrak{X}$. The character $\chi^\rho$ afforded by $\mathfrak{X}^\rho$ satisfies
$\chi^\rho(\sigma_s) = \chi(\sigma_s)^\rho$. We say that $\chi^\rho$ is the \textit{algebraic conjugate} of $\chi$ by $\rho$
, or $\chi$ and $\chi^\rho$ are \textit{algebraically conjugate}.

\subsection{Strongly normal closed subsets of prime index}\label{sec:primeindex}
Let $(X, S)$ be an association scheme and $T$ a strongly normal closed subset of $S$.
For $x \in X$, the adjacency algebra of the subscheme $(X,S)_{xT}$ is isomorphic to $\mathbb{C}T$ (see \cite{zies} for the definition).
So we can regard $\mathbb{C}T$ as the adjacency algebra. Put
\[G:=S/\!/T. \]
Then the adjacency algebra $\mathbb{C}S$ is a $G$-graded algebra.
Based on the result given in \cite{hanaki1}, we build up this subsection.

Let $\varphi$ be an irreducible character of $T$ and $L$ an irreducible right $\mathbb{C}T$-module affording $\varphi$.
The \textit{induction} of $L$ to $S$ is defined by
\[ L^S = L \otimes_{\mathbb{C}T} \mathbb{C}S = \bigoplus_{s^T \in S/\!/T} L \otimes \mathbb{C}(TsT).\]
Define the \textit{support} $\mathrm{Supp}(L^S)$ of $L^S$ by
\[\mathrm{Supp}(L^S)=\{ s^T  \in S/\!/T \mid L \otimes \mathbb{C}(TsT)\neq0 \},\]
and the \textit{stabilizer} $G\{L\}$ of $L$ in $G$ by
\[ G\{L\}=\{s^T \in S/\!/T \mid L \otimes \mathbb{C}(TsT)\cong L ~(\text{as}~ \text{$\mathbb{C}T$-modules}) \} .\]
Then $G\{L\}$ is a subgroup of $G$ and $\mathrm{Supp}(L^S)$ is a union of left cosets of $G\{L\}$ in $G$.

Under the assumption that $T$ is a strongly normal closed subset of $S$, we have the following result.
\begin{thm}[see \cite{hanaki}]\label{thm:irreducible2}
For any irreducible $\mathbb{C}T$-module $L$ and $s \in S$, $L\otimes_{\mathbb{C}T} \mathbb{C}(TsT)$ is an irreducible $\mathbb{C}T$-module or $0$.
\end{thm}



For a character $\varphi$ of $T$ afforded by an irreducible $\mathbb{C}T$-module $L$, we write $\varphi^S$ for the induced character of $\varphi$ to $S$, and for a character $\chi$ of $S$ we write $\chi_{_T}$ for the restriction of $\chi$ to $T$. We put $\mathrm{Supp}(\varphi^S):= \mathrm{Supp}(L^S)$ and $G\{ \varphi \}:= G\{L\}$.

\begin{thm}[see \cite{hanaki1}]\label{thm:key}
Let $(X, S)$ be an association scheme and $T$ a strongly normal closed subset of $S$.
Suppose $G = S/\!/T$ is the cyclic group of prime order $p$.
Then for $\chi \in \mathrm{Irr}(S)$, one of the following statements holds:
\begin{enumerate}
\item[(1)] $\chi_{_T} \in \mathrm{Irr}(T)$ and $(\chi_{_T})^S = \sum_{i=1} ^p \chi \xi_i$, where $\mathrm{Irr}(G)=\{ \xi_i \mid 1 \leq i \leq p\}$;
\item[(2)] $\chi(\sigma_s) = 0$ for any $s \in S \setminus T$ and $\chi_{_T}$ is a sum of at most $p$ distinct irreducible characters.
If $\psi$ is an irreducible constituent of $\chi_{_T}$, then $\psi^S=\chi$.
\end{enumerate}
\end{thm}

\begin{obs}\label{remark:2}
\emph{In Theorem \ref{thm:key}(1), $G\{\chi_{_T}\} = G$ and $\chi \xi_i \in \mathrm{Irr}(S)$.
In Theorem \ref{thm:key}(2), $G\{\psi \} = \{ 1_{_G} \}$ and $\chi_{_T}$ is a sum of $k$ distinct irreducible characters of $T$, where $k= |\mathrm{Supp}(\psi^S)|$.}
\end{obs}

\begin{thm}[see \cite{hanaki}]\label{thm:key1}
Let $(X,S)$ be a scheme, $T$ a closed subset of $S$ and $\psi \in \mathrm{Irr}(T)$. Then
\[ (n_{_S}/n_{_T})m_\psi = \sum_{\chi \in Irr(S)} (\psi^S, \chi)_{_S} m_\chi ,\]
where $(\psi^S, \chi)_{_S} = a_\chi$ for $\psi^S = \sum_{\varphi \in \mathrm{Irr}(S)} a_\varphi \varphi$, and $m_\psi$ and $m_\chi$ are the multiplicities on $T$ and $S$, respectively.
\end{thm}

\subsection{Partial linear spaces and difference families}\label{sec:pardf}

In \cite[Section 4]{mp}, it is shown that any reduced Klein configuration defines an incidence structure. By a parallel argument, we will obtain a partial linear space over a scheme under certain assumption.

In geometric terms, a \textit{partial linear space} is a pair $(\mathcal{P},\mathcal{L})$, where $\mathcal{P}$ is a set of points and $\mathcal{L}$ is a family of subsets of $\mathcal{P}$, called \textit{lines}, such that every pair of points is contained in at most one line and every line contains at least two points. Moreover, $(\mathcal{P}, \mathcal{L})$ is called a \textit{linear space} if every pair of points is contained in exactly one line.

Let $(X,S)$ be an association scheme with $T:=\Or(S) \cong C_p \times C_p$ and let
$\{ X_i \mid i=1, 2, \dots, m  \}$ denote the set of cosets of T in $S$.
Then we define the \textit{left stabilizer} of $s$
\[L(s):=\{ t \in T \mid ts=s \}.\]

Note that $L(s) = L(s')$ for $s' \in TsT$.
For given $i, j$ and $s \in S$ with $s \cap (X_i\times X_j) \neq \varnothing$, we define $L_{ij}(s):= L(s)$.
Since $L_{ij}(s) = L_{ij}(s')$ for $s' \in TsT$, we write $L_{ij}$ for $L_{ij}(s)$.
For $\{ 1_{_X} \} < M < T$, we put $L_i(M):= \{ i \} \cup \{ j \mid L_{ij}=M \}$.

Define an incidence structure $(\mathcal{P}, \mathcal{L})$ as follows:
\begin{enumerate}
\item $\mathcal{P}:= \{ i \mid i= 1, 2, \dots, m  \}$;
\item $\mathcal{L}:= \{ L_i(M) \mid \{1_{_X}\} < M < T, |L_i(M)|\neq 1, 1 \leq i \leq m \}$.
\end{enumerate}
An \textit{automorphism} $\tau$ of $(\mathcal{P}, \mathcal{L})$ is a permutation of $\mathcal{P}$ such that $\mathcal{L}^\tau = \mathcal{L}$.
We denote by $\mathrm{Aut}(\mathcal{P}, \mathcal{L})$ the set of automorphisms of $(\mathcal{P}, \mathcal{L})$.

\begin{lem}\label{lem:lemsss}
Let $(X,S)$ be an association scheme such that $T:= \mathbf{O}_\mathrm{\theta}(S) = \Or(S) \cong C_p \times C_p$.
For $s_1, s_2 \in S \setminus T$, if $L(s_1) = L(s_2) \cong C_p$,
then either $s_2^* s_1 \subseteq T$ or $s_2^* s_1 = s_3$ for some $s_3 \in S \setminus T$.
\end{lem}
\begin{proof}
Note that $n_{st}=n_{ts}=n_s$ and $n_s > 1$ for each $s \in S \setminus T$ and $t \in T$.
Since $T$ acts transitively on $TsT$ and $n_{_T} = p^2$, the length of $T$-orbit is $1$ or $p$. Thus,
$n_s$ is $p$ or $p^2$ for $s \in S \setminus T$.

Suppose $s_2^* s_1 \nsubseteq T$. Then $\sigma_{s_2^*}\sigma_{s_1}$ is a linear combination of at most $p$ adjacency matrices.
Since $s_2^* s_1 = s_2^* t s_1$ for each $t \in L(s_1)$, $\sigma_{s_2^*}\sigma_{s_1} = p \sigma_{s_3}$ for some $s_3 \in S \setminus T$.
\end{proof}

Using the key idea in the proof of \cite[Lemma 4.3]{mp}, we show the following theorem.
\begin{thm}\label{thm:psp}
Let $(X,S)$ be an association scheme such that $T:= \mathbf{O}_\mathrm{\theta}(S) = \Or(S) \cong C_p \times C_p$.
Then $(\mathcal{P}, \mathcal{L})$ is a partial linear space such that any point belongs to at most $p+1$ lines,
and $\mathrm{Aut}(X/T,S/\!/T)$ is a subgroup of $\mathrm{Aut}(\mathcal{P}, \mathcal{L})$.
\end{thm}
\begin{proof}
If $\{ j \in \mathcal{P} \mid \{1_{_X}\} < L_{ij} < T\} = \varnothing$ for each $i \in \mathcal{P}$,
then $(\mathcal{P}, \mathcal{L})$ is a partial linear space without lines.
Suppose that $\{ j \in \mathcal{P} \mid \{1_{_X}\} < L_{ij} < T\} \neq \varnothing$ for some $i \in \mathcal{P}$.
By definition, any line contains at least two points.

First, we claim that for $i \neq j$,
\begin{equation}\label{B}
 j \in L_i(M) \Rightarrow L_i(M) = L_j(L_{ji}).
\end{equation}

To show $L_i(M) \supseteq L_j(L_{ji})$, let $k \in L_j(L_{ji})$.
If $k$ is $i$ or $j$, then clearly $k \in L_i(M)$.
Assume $k \neq i,j$. Then $L_{jk}=L_{ji}$ so that
$L_{jk}=L_{jk}(s_1)$ and $L_{ji}=L_{ji}(s_2)$ for some $s_1, s_2 \in S \setminus T$.
Since $k \neq i$, we have $s_2 \notin Ts_1T$. By Lemma \ref{lem:lemsss}, $s_2^*s_1 = s_3$ for some $s_3 \in S \setminus T$.
Since $L(s_2^*)=L(s_3)\cong C_p$, we have $L_{ij}=L_{ik}$. Thus, by definition $k \in L_i(L_{ij})=L_i(M)$.

Since $i \in L_j(L_{ji})$, by the same argument as above $L_i(L_{ij}) \subseteq L_j(L_{ji})$. Thus, $L_i(M) = L_j(L_{ji})$.

Now we prove that for any $k, k' \in \mathcal{P}$ and for $\{ 1_{_X}\} < H, H' < T$,
\[ |L_k(H) \cap L_{k'}(H')|\geq 2 \Rightarrow L_k(H) = L_{k'}(H'). \]
Suppose that $i, j \in L_k(H) \cap L_{k'}(H')$ and $i\neq j$.
We divide our consideration into two cases.
\begin{enumerate}
\item $i \neq k$.

By the definition, we have $L_{ki}=H$. Also, by (\ref{B}) $L_k(H)=L_k(L_{ki})=L_i(L_{ik})$.
Since $j \in L_k(H)= L_i(L_{ik})$, we have $L_{ik}=L_{ij}$.
It follows from (\ref{B}) that $L_i(L_{ik})= L_i(L_{ij})= L_j(L_{ji})$.
Thus, $L_k(H)= L_i(L_{ij})= L_j(L_{ji})$.

\item $i=k$.

Since $i\neq j$, we have $j\neq k$. By argument of case(i), we get $L_k(H)= L_i(L_{ij})= L_j(L_{ji})$.
\end{enumerate}

Similarly, we consider the following cases: $i \neq k'$ and $i=k'$. We can get $L_{k'}(H')= L_i(L_{ij})= L_j(L_{ji})$.
Therefore, $L_k(H) = L_{k'}(H')$, i.e., any two distinct points belong to at most one line.
Since $T$ has exactly $p+1$ proper subgroups, any point belongs to at most $p+1$ lines.

Finally, we show that
\[\mathrm{Aut}(X/T,S/\!/T) \leq\mathrm{Aut}(\mathcal{P}, \mathcal{L}).\]
It suffices to prove that $\phi(L_i(M))$ belongs to $\mathcal{L}$ for any $\phi$ and $L_i(M) \in \mathcal{L}$.
For convenience, we identify $X/T$ with $\mathcal{P}$.
Note that $L_i(M)^\phi = \{i^\phi \} \cup \{ j^\phi \mid L_{ij}=M \}$.
Since $L_{ij}=L_{i^\phi j^\phi}$, $L_i(M)^\phi = L_{i^\phi}(M) \in \mathcal{L}$.
\end{proof}

\begin{cor}\label{cor:lem32}
Under the assumption of Theorem \ref{thm:psp},
if $\{ n_s \mid s \in S \} = \{1, p \}$, then $(\mathcal{P}, \mathcal{L})$ is a linear space.
Moreover, $\frac{n_{_S}}{n_{_T}}  \leq p^2 + p + 1$.
\end{cor}
\begin{proof}
Clearly, $(\mathcal{P}, \mathcal{L})$ is a partial linear space. By the assumption, for given $i \neq j$ $L_{ij}$ is always a proper subgroup. Thus, every pair of points is contained in exactly one line.

We claim that each line contains at most $p+1$ points.
If not, any point which is not in the line containing at least $p+2$ points would belong to at least $p+2$ lines, a contradiction by Theorem \ref{thm:psp}.
Therefore, for a fixed point, that point belongs to at most $p+1$ lines, each of which containing at most $p+1$ points.
This implies that $|\mathcal{P}| \leq (p+1)^2 -p = p^2 +p +1$.

\end{proof}

Let $G$ be a finite group of order $v$ acting on itself by right multiplication, and let $\mathbf{S}=\{B_1, \dots, B_s\}$
be a family of subsets of $G$ satisfying
\[ |\{ g \in G \mid gB_j=B_j \}|=1 ~~\text{for}~ 1 \leq j \leq s, \]
and put $K:= \{ |B_j| \mid 1 \leq j \leq s\}$.
If every non-identity element of $G$ occurs $\lambda$ times in $\bigcup_{1 \leq j \leq s} \Delta B_j$, where
$\Delta B_j := \{ ab^{-1} \mid a, b \in B_j, a \neq b \}$, then $\mathbf{S}$ is called a $(v, K, \lambda)$-\textit{difference family} in $G$.

\begin{prop}[\cite{beth}]\label{prop:prop1}
Let $G$ be a finite group of order $v$, and let $\mathbf{S}=\{B_1, \dots, B_s\}$
be a family of subsets of $G$ satisfying $|\{ g \in G \mid gB_j=B_j \}|=1$ for $1 \leq j \leq s$.
Then $\mathbf{S}$ is a $(v, K, 1)$-difference family if and only if $(\mathcal{P}, \mathcal{L})$ is a linear space,
where $\mathcal{P}:=G$ and $\mathcal{L}:=\{ B_jg \mid g \in G, 1 \leq j \leq s \}$.
\end{prop}

\begin{obs}
\emph{By Corollary \ref{cor:lem32} and Proposition \ref{prop:prop1}, an association scheme $(X,S)$ such that $\mathbf{O}_\mathrm{\theta}(S) = \Or(S) \cong C_p \times C_p$ and and $\{ n_s \mid s \in S \} = \{1, p \}$ induces a difference family with $\lambda =1$.}
\end{obs}


\section{Thin residues as the elementary abelian $p$-group of rank two}\label{sec:main1}

Throughout this section we assume that $(X,S)$ is a scheme such that
\begin{enumerate}
\item $\Or(S)$ is the elementary abelian group of order $p^2$,
\item $S/\!/\Or(S)$ is the cyclic group of order $q$, where $p$ and $q$ are primes.
\end{enumerate}

For convenience, we denote $S/\!/T$ by $G$ and $\Or(S)$ by $T$.

\begin{lem}\label{lem:lem33}
For all $\varphi \in \mathrm{Irr}(T) \setminus \{1_{_T}\}$ and $s \in S \setminus T$, the following are equivalent:
\begin{enumerate}
\item $L(s)= \mathrm{Ker}\varphi$;
\item $e_\varphi \sigma_s \neq 0$;
\item $s^T\in \mathrm{Supp}(\mathbb{C}e_\varphi^S)$, i.e., $\mathbb{C}e_\varphi \otimes \mathbb{C}(TsT) \neq 0$.
\end{enumerate}
\end{lem}
\begin{proof}

Note that $T = \langle t \rangle \mathrm{Ker}\varphi$ for some $t \in T \setminus \mathrm{Ker}\varphi$ and $e_\varphi = (\frac{1}{p} \sum_{i \in \mathbb{Z}_p} \varepsilon^i \sigma_t^i)(\frac{1}{p} \sum_{t' \in \mathrm{Ker}\varphi} \sigma_{t'})$ for some primitive $p$-th root of unity $\varepsilon$.

First, we show that $L(s)= \mathrm{Ker}\varphi$ if and only if $e_\varphi \sigma_s \neq 0$.
Assume $L(s)= \mathrm{Ker}\varphi$.
Since $\langle t \rangle$ acts regularly on $TsT$,
\[ e_\varphi \sigma_s = (\frac{1}{p} \sum_{i \in \mathbb{Z}_p} \varepsilon^i \sigma_t^i) \sigma_s = (\frac{1}{p} \sum_{i \in \mathbb{Z}_p} \varepsilon^i \sigma_{t^is}) \neq 0. \]

Assume $e_\varphi \sigma_s \neq 0$. If $L(s) \neq \mathrm{Ker}\varphi$, then since $\langle t \rangle$ and $\mathrm{Ker}\varphi$ act regularly on $TsT$,
\[e_\varphi \sigma_s = (\frac{1}{p} \sum_{i \in \mathbb{Z}_p} \varepsilon^i \sigma_t^i)(\frac{1}{p} \sum_{t' \in \mathrm{Ker}\varphi} \sigma_{t'})\sigma_s = (\frac{1}{p} \sum_{i \in \mathbb{Z}_p} \varepsilon^i \sigma_t^i)(\frac{1}{p} \sum_{s' \in TsT} \sigma_{s'})= (\frac{1}{p^2} \sum_{i \in \mathbb{Z}_p} \varepsilon^i )(\sum_{s' \in TsT} \sigma_{s'}) =0, \]
a contradiction.

Next, we show that $\mathbb{C}e_\varphi \otimes \mathbb{C}(TsT) = 0$ if and only if $e_\varphi \sigma_s = 0$.
By Theorem \ref{thm:irreducible2},
the dimension of $\mathbb{C}e_\varphi \otimes \mathbb{C}(TsT)$ is at most $1$. Thus, $\mathbb{C}e_\varphi \otimes \mathbb{C}(TsT)= 0$ if and only if
$e_\varphi \otimes \sigma_s = 0$.

Since $e_\varphi \otimes \sigma_s = e_\varphi \otimes e_\varphi \sigma_s$ by definition of tensor product,
$e_\varphi \otimes \sigma_s = 0$ if and only if $e_\varphi \sigma_s = 0$. This completes the proof.
\end{proof}

\begin{lem}\label{lem:34}
For all $\varphi \in \mathrm{Irr}(T) \setminus \{1_{_T}\}$ and $s \in S \setminus T$,
$s^T\in G\{\mathbb{C}e_\varphi\}$ if and only if $\mathrm{Ker} \varphi=L(s)=L(s^\ast)$.
\end{lem}
\begin{proof}
Suppose that $1_{_X}^T\ne s^T\in G\{\mathbb{C}e_\varphi\}$.
This implies that $e_\varphi \sigma_s$ is left left invariant by the right multiplication of $e_\varphi$,
i.e.,
\[(e_\varphi \sigma_s)e_\varphi=e_\varphi \sigma_s.\]
Since $e_\varphi \sigma_s$ is nonzero, so $\sigma_s e_\varphi$ is.
Taking the adjoint operator on $\sigma_s e_\varphi\ne 0$ we obtain that
$(e_\varphi)^\ast(\sigma_{s})^\ast \ne 0$.
Note that $(\sigma_s)^\ast=\sigma_{s^\ast}$ and $(e_\varphi)^\ast=e_\varphi$
since $e_\varphi^\ast$ is a primitive idempotent of $\mathbb{C}T$ with $e_\varphi (e_\varphi)^\ast\ne 0$.
Thus, we have $ e_\varphi\sigma_{s^\ast} \ne 0$.
Applying Lemma~\ref{lem:lem33} for both $e_\varphi \sigma_s\ne 0$ and $e_\varphi \sigma_{s^\ast}  \ne 0$ we
obtain that $\mathrm{Ker} \varphi=L(s)=L(s^\ast)$.

Suppose $\mathrm{Ker} \varphi=L(s)=L(s^\ast)$.
By Lemma~\ref{lem:lem33} we have
\[\mbox{$e_\varphi \sigma_s\ne 0$ and $e_\varphi \sigma_{s^\ast}\ne 0$.}\]
This implies that $\sigma_s e_\varphi\ne 0$ since $(\sigma_s)^\ast=\sigma_{s^\ast}$ and $(e_\varphi)^\ast=e_\varphi$.
Therefore, the right multiplication of $e_\varphi$ to $\mathbb{C}e_\varphi \otimes \mathbb{C}(TsT)$
acts identically on it. It follows that $\mathbb{C}e_\varphi \otimes \mathbb{C}(TsT)\simeq \mathbb{C}e_\varphi$
as right $\mathbb{C}T$-modules, and hence, $s^T\in G\{\mathbb{C}e_\varphi\}$.
\end{proof}

\begin{lem}\label{lem:35}
The following are equivalent:
\begin{enumerate}
\item $(X,S)$ is commutative;
\item For each $\varphi\in \mathrm{Irr}(T)\setminus \{1_{_T}\}$, $\mathrm{Supp}(\mathbb{C}e_\varphi^S)=\{1_{_X}^T\}$;
\item For each $s\in S\setminus T$ we have $|TsT|=1$.
\end{enumerate}
\end{lem}
\begin{proof}
Suppose (i) and $\mathrm{Sup}p(\mathbb{C}e_\varphi^S)\ne \{1_{_X}^T\}$ for some $\varphi\in \mathrm{Irr}(T)\setminus \{1_{_T}\}$.
Then, by definition, $s^T\in \mathrm{Supp}(\mathbb{C}e_\varphi^S)$ for some $s\in S\setminus T$.
Since $(X,S)$ is commutative, we have
$ss^\ast=s^\ast s$, or equivalently, $L(s)=L(s^\ast)$.
It follows from Lemma~\ref{lem:34} that
$s^T\in G\{\mathbb{C}e_\varphi\}$.
Note that $G(\mathbb{C}e_\varphi\}$ is a subgroup of $G$ and $G$ is a cyclic group of prime order, it follows that
$G\{\mathbb{C}e_\varphi\}=G$.
By Lemma~\ref{lem:34}, $\mathrm{Ker} \varphi=L(u)=L(u^\ast)$ for each $u\in S\setminus T$.
Since $L(u)=uu^\ast$ for each $u\in S$, it follows that
\[\mathrm{Ker} \varphi<T=\Or(S)=\gn{uu^\ast\mid u\in S}=\mathrm{Ker} \varphi,\] a contradiction.

Suppose (ii) and $|TsT|\ne 1$ for some $s\in S\setminus T$.
Then $|TsT|>1$ and $ss^\ast$ is a closed subset of $T$ with valency $p$,
otherwise, $n_s=1$, implying that $S$ is thin and $n_{_T}=1$, a contradiction.
Take $\psi \in \mathrm{Irr}(T)\setminus \{1_{_T}\}$ such that $\mathrm{Ker} \psi=L(s)$.
Then, by Lemma~\ref{lem:lem33}, $s^T\in \mathrm{Supp}(\mathbb{C}e_\psi^S)$, which contradicts (ii).

Suppose (iii). Let $s,u\in S\setminus T$.
Since $S\Qd T$ is abelian, $T(su)T=T(us)T$.
If $us\subseteq T$, then $\sigma_s\sigma_u=p^2\sigma_{_T}=\sigma_u\sigma_s$.
If not, then $us$ and $su$ are singletons by (iii), and hence $\sigma_s\sigma_u=\sigma_u\sigma_s$.
Since $T$ acts trivially on $S\setminus T$ by the right and left multiplication respectively,
it follows that $(X,S)$ is commutative.
\end{proof}

\begin{lem}\label{lem:36}
For each $\varphi\in \mathrm{Irr}(S)\setminus \{1_{_T}\}$,
if $\mathrm{Supp}(\mathbb{C}e_\varphi^S)\ne \{1_{_X}^T\}$, then $\varphi^S\in \mathrm{Irr}(S)$.
\end{lem}
\begin{proof}
Suppose $1_{_X}^T\ne s^T\in \mathrm{Supp}(\mathbb{C}e_\varphi^S)$.
Then we conclude from Lemma~\ref{lem:35} that
$(X,S)$ is not commutative, and from Lemma~\ref{lem:lem33} that $L(s)=\mathrm{Ker} \varphi$.
Note that $\varphi^S(\sigma_{1_X})= \mathrm{dim} \mathbb{C}e_\varphi \otimes \mathbb{C}S$,
\begin{equation}\label{eq35}
\mathbb{C}e_\varphi \otimes \mathbb{C}S=\bigoplus_{s^T\in S\Qd T} \mathbb{C}e_\varphi \otimes \mathbb{C}(TsT)
\mbox{ and }
\mathrm{dim}\mathbb{C}e_\varphi \otimes \mathbb{C}(TsT)\in \{0,1\}.
\end{equation}

We claim that $\mathrm{Supp}(\mathbb{C}e_\varphi^S)\ne G$.
Otherwise, $L(u)=\mathrm{Ker} \varphi$ for each $u\in S\setminus T$ by Lemma~\ref{lem:lem33},
which implies that $L(s)<T=\Or(S)=\gn{uu^\ast\mid u\in S}=L(s)$, a contradiction.

Let $\chi\in \mathrm{Irr}(S)$ be an irreducible constituent of $\varphi^S$.
Then $\varphi$ is an irreducible constituent of $\chi_{_T}$.
So (1) or (2) of Theorem \ref{thm:key} holds.
But, (1) never holds, since
$\varphi^S(\sigma_{1_X})<p$ by the claim and (\ref{eq35}).
This implies that $\varphi^S\in \mathrm{Irr}(S)$.
\end{proof}

\begin{thm}\label{thm:main3}
Suppose that $(X,S)$ is an association scheme which satisfies the following:
\begin{enumerate}
\item $T :=\Or(S)$ is the elementary abelian group of order $p^2$;
\item $S/\!/T$ is the cyclic group of order $q$;
\end{enumerate}
Then $\{ \chi(\sigma_{1_X}) \mid \chi \in \mathrm{Irr}(S), \chi(\sigma_{1_X})\neq 1 \} = \{ |l| \mid l \in \mathcal{L} \}$, where $\mathcal{L}$ is the set of lines on $X/T$ as in Section~2.
\end{thm}
\begin{proof}
Assume $(X, S)$ is commutative.
Then by Lemma~\ref{lem:35} $|TsT|=1$ for each $s \in S \setminus T$.
This implies that $(\mathcal{P}, \mathcal{L})$ is a partial linear space without lines.
Since every irreducible character of $S$ is linear, this case is done.

Assume $(X,S)$ is not commutative.
Then, by Lemma~\ref{lem:35}, there exists $s\in S\setminus T$ such that $|TsT|=p$.
Let $\chi \in \mathrm{Irr}(S)$ such that $\chi(\sigma_{1_X}) \neq 1$.
Since $\chi(\sigma_{1_X})>1$, we conclude from Theorem \ref{thm:key} that
$\chi=\psi^S$ for some $\psi\in \mathrm{Irr}(T)$.
Note that the equation obtained from (\ref{eq35}) by replacement of $\varphi$ by $\psi$ holds.
Since $\mathrm{dim}(\mathbb{C}e_\psi)\otimes \mathbb{C}(TsT)=1$, it follows from
Lemma \ref{lem:lem33} that
\[\psi^S(\sigma_{1_X}) =|\{s^T\in G\setminus \{1_{_X}\} \mid L(s)=\mathrm{Ker}\psi \}| + 1,\]
which is equal to the size of a line by definition of $L_i(L(s))$.

Conversely, let $L_i (M) \in \mathcal{L}$, where $M$ is a closed subset of $T$ with valency $p$.
Then we can take $\psi\in \mathrm{Irr}(T)$ such that $\mathrm{Ker} \psi=M$.
By Lemma~\ref{lem:36}, $\psi^S\in \mathrm{Irr}(S)$.
Therefore, we conclude from Lemma~\ref{lem:lem33} that $1<|L_i(M)|=\psi^S(\sigma_{1_X})$ as desired.
\end{proof}

\section{Proofs of the main theorems}\label{sec:main2}
For each $\iota \in \mathrm{Aut}(T)$, we define $\tilde{\iota} : S \rightarrow S$ by
$s \mapsto s ~\text{if}~ s \in S \setminus T ~~\text{and}~~ s \mapsto s^\iota ~\text{if}~ s \in T$.

\begin{flushleft}
\textbf{Proof of Theorem~\ref{thm:main1}}
\end{flushleft}
Clearly, $1_{_X}'= 1_{_X}$ and $(s')^\ast \in S'$ for each $s' \in S'$. It is enough to show that, for all $u'$, $v'$, $w' \in S'$, $|xu'\cap y(v')^\ast|$ is constant whenever $(x,y)\in w'$.

We divide our consideration into four cases depending on $|\{u, v, w \} \cap T|$.
\begin{enumerate}
\item $|\{u, v, w \} \cap T| = 3$. For each $(x,y) \in w'$, if $x, y \in X_i$, then
\[ |xu'\cap y(v')^\ast|=|xu^{\iota_{_i}} \cap y(v^{\iota_{_i}})^\ast|=c_{u^{\iota_{_i}}v^{\iota_{_i}}}^{w^{\iota_{_i}}}.\]
Since $\iota_{_i} \in \mathrm{Aut}(T)$, we have $c_{uv}^{w} = c_{u^{\iota_{_i}}v^{\iota_{_i}}}^{w^{\iota_{_i}}}$. This implies that $c_{u'v'}^{w'}$ is well-defined.
\item $|\{u, v, w \} \cap T| = 2$. First, we assume $w \in S \setminus T$.
Then $u, v \in T$ and $w'=w$.
For each $(x,y) \in w'$, if $(x, y) \in X_i \times X_j$ for some $i \neq j$, then
\[ |xu'\cap y(v')^\ast|=|xu^{\iota_{_i}} \cap y(v^{\iota_{_j}})^\ast| = c_{u^{\iota_{_i}}v^{\iota_{_j}}}^{w}. \]
Since $T$ is a closed subset of $S$, $c_{u^{\iota_{_i}}v^{\iota_{_j}}}^{w}=0$. Thus, $c_{u'v'}^{w'}=0$ and $c_{u'v'}^{w'}$ is well-defined.
Now we assume $w \in T$. Then $v \in S \setminus T$ or $u \in S \setminus T$.
Since $n_w c_{uv}^{w} = n_v c_{w^\ast u^\ast}^{v^\ast} = n_u c_{v w^\ast}^{u^\ast}$, we obtain the same conclusion as in the case when $w$ belongs to $S \setminus T$.

\item $|\{u, v, w \} \cap T| = 1$. First, we assume that $u \in T$ and $v, w \in S \setminus T$.
For each $(x, y) \in w'$, if $x \in X_i$, then
\[ |xu'\cap y(v')^\ast|=|xu^{\iota_{_i}} \cap yv^\ast|=c_{u^{\iota_{_i}}v}^{w}. \]
Since $v, w \in S \setminus T$ and $\tilde{\iota}_{_i} \in \mathrm{Aut}(S)$, we have $c_{u^{\iota_{_i}}v}^{w} = c_{u^{\tilde{\iota}_{_i}}v^{\tilde{\iota}_{_i}}} ^{w^{\tilde{\iota_{_i}}}}=c_{uv}^{w}$. Thus, $c_{uv}^{w} = c_{u'v'}^{w'}$ and $c_{u'v'}^{w'}$ is well-defined.
Similarly, we obtain the same conclusion as in the case when $v$ or $w$ belongs to $T$.
\item $|\{u, v, w \} \cap T| = 0$. It is clear that $c_{u'v'}^{w'}$ is well-defined and $c_{uv}^{w} = c_{u'v'}^{w'}$.
\end{enumerate}
This completes the proof.
\qed

\begin{flushleft}
\textbf{Proof of Theorem~\ref{thm:main2}}
\end{flushleft}
Let $\iota$ be an element of $\mathrm{Aut}(T)$ fixing each element of $H$. For convenience, we denote $s^{\tilde{\iota}}$ by $\tilde{s}$.
We will show that $c_{uv}^{w} = c_{\tilde{u}\tilde{v}}^{\tilde{w}}$ for all $u, v, w \in S$.

We divide our consideration into four cases depending on $|\{u, v, w \} \cap T|$.
\begin{enumerate}
\item $|\{u, v, w \} \cap T| = 3$. Since $\iota$ is an element of $\mathrm{Aut}(T)$, we have $c_{\tilde{u}\tilde{v}}^{\tilde{w}}=c_{u^{\iota}v^{\iota}}^{w^{\iota}}=c_{uv}^{w}$.
\item $|\{u, v, w \} \cap T| = 2$. Since $T$ is a closed subset of $S$, $c_{uv}^{w}=0=c_{\tilde{u}\tilde{v}}^{\tilde{w}}$. Thus we have $c_{\tilde{u}\tilde{v}}^{\tilde{w}}=c_{uv}^{w}$.
\item $|\{u, v, w \} \cap T| = 1$. First, we assume that $u \in T \setminus H$.
    If $wv^* \nsubseteq T$, then clearly $c_{uv}^{w} = 0 = c_{\tilde{u}\tilde{v}}^{\tilde{w}}$. Now assume $wv^* \subseteq T$.
Let $(x, y) \in \tilde{w} = w$. Note that $xu$ is a union of cosets of $H$ in $xT$ by the second condition, and
\[y\tilde{v}^\ast = yv^\ast = yv^\ast \cap yv^\ast T = yv^\ast \cap yv^\ast H = \bigcup_{i=1}^m (yv^\ast \cap x_i H),\]
where $\{ x_1, x_2, \dots, x_m \}$ is a transversal of $H$ in $yv^\ast T$ such that $(y, x_i) \in v^\ast$ for each $i$.

Since $|yv^\ast \cap x_i H| = c_{v^\ast H}^{v^\ast}$ does not depend on the choice of $i$, it follows that
\[ |xu \cap yv^\ast| = |\bigcup_{i=1}^m  (xu \cap (yv^\ast \cap x_i H))| = (\text{the number of cosets of} ~H~ \text{in} ~xu) \times c_{v^\ast {_H}}^{v^\ast} = \frac{n_u}{n_{_H}}c_{v^\ast {_H}}^{v^\ast}. \]
Since $n_u = n_{\tilde{u}}$, $c_{uv}^w$ is equal to $c_{\tilde{u}v}^{w}$.
Thus, since $\tilde{w}=w$ and $\tilde{v}=v$, we have $c_{uv}^{w} = c_{\tilde{u}\tilde{v}}^{\tilde{w}}$.
Next, we assume that $u \in H$. Then it is clear that $c_{uv}^{w} = c_{\tilde{u}\tilde{v}}^{\tilde{w}}$.
Since $n_w c_{uv}^{w} = n_v c_{w^\ast u^\ast}^{v^\ast} = n_u c_{v w^\ast}^{u^\ast}$, we obtain the same conclusion as in the case when $v$ or $w$ belongs to $T$.
\item $|\{u, v, w \} \cap T| = 0$. By the definition of $\tilde{\iota}$, clearly $c_{uv}^{w} = c_{\tilde{u}\tilde{v}}^{\tilde{w}}$.
\end{enumerate}
This completes the proof.
\qed

\begin{obs}\label{p3remark}
\emph{Suppose that $(X,S)$ is a $p$-scheme of order $p^3$ and $T:=O^\theta(S)$ is not a group.
Put $H:=\{ \ t \in T \mid n_t = 1 \}$. Then $T$ and $H$ satisfy the conditions in Theorem \ref{thm:main2}.}
\end{obs}

The following is an example about Remark \ref{p3remark}.

\begin{example}\label{example}
Let $V:=\mathbb{F}_p^3$ and $G$ be a group generated by $\{ \tau_y, \iota \mid y \in V \}$, where
$\tau_y$ and $\iota$ are functions from $V$ to $V$ defined by $\tau_y(x)=x + y$ and $\iota((x_1, x_2, x_3))=((x_1 + x_2, x_2 + x_3, x_3))$, respectively. Then the orbitals of $G$ on $V$ are
\[ (0, 0)^G, (0, e_1)^G, (0, 2e_1)^G, \dots, (0, (p-1)e_1)^G,  \]
\[ (0, e_2)^G, (0, 2e_2)^G, \dots, (0, (p-1)e_2)^G, \]
\[ (0, e_3)^G, (0, e_3 + e_1)^G, \dots, (0, e_3 + (p-1)e_1)^G,  \]
\[ (0, 2e_3)^G, (0, 2e_3 + e_1)^G, \dots, (0, 2e_3 + (p-1)e_1)^G,  \]
\[.............................................................................,\]
\[ (0, (p-1)e_3)^G, (0, (p-1)e_3 + e_1)^G, \dots, (0, (p-1)e_3 + (p-1)e_1)^G,  \]
where $\{ e_1, e_2, e_3 \}$ is the standard basis of $V$.

Let $(X, S)$ be the Schurian scheme $(V,\mathcal{R}_G)$ induced by $G$ on $V \times V$.
Put
\[H:=\{ (0, 0)^G, (0, e_1)^G, (0, 2e_1)^G, \dots, (0, (p-1)e_1)^G \}\]
and
\[T:=H\cup \{ (0, e_2)^G, (0, 2e_2)^G, \dots, (0, (p-1)e_2)^G \}.\]
Note that $G \cong V \rtimes\langle \iota \rangle$, $H \cong C_p$ and $T \cong C_p \wr C_p$ as a wreath product of schemes.
Then one can check that $T$ and $H$ satisfy the conditions in Theorem \ref{thm:main2}.
Further, we can observe the following:
For $s \in S \setminus T$ and $t \in T \setminus H$,
\begin{enumerate}
\item there exists a unique $s_1 \in S \setminus T$ such that $c_{ss}^{s_1} = 1$;
\item $ss^\ast = s^\ast s = \{1_{_X}\} \cup (T \setminus H)$;
\item $|tt|=1$.
\end{enumerate}
\end{example}
Note that (ii) implies that $c_{s^\ast s}^t = 1$ for each $t \in T \setminus H$, and hence $c_{t s^\ast}^{s^\ast} = 1$ since $n_s = n_t = p$.
We will denote by $t'$ the element in $tt$ which is unique by (iii).

\begin{thm}\label{thm:nonschurian}
For each odd prime $p$, there exists a non-Schurian scheme with the same intersection numbers as the Schurian scheme $(X,S)$ given in Example \ref{example}.
\end{thm}
\begin{proof}
Fix $s_1 \in S \setminus T$ and $t_1 \in T \setminus H$.
Let $x \in X$ and $y_1 \in xs_1$. Then there exists a unique $y_2 \in y_1 t_1  \cap  xs_1$ since $c_{t_1 s_1^\ast}^{s_1^\ast} = 1$.
Inductively, we can define $y_{i+1} \in y_i t_1 \cap   xs_1$ for $i = 1, \dots, p-1$.
Note that $(y_p, y_1) \in t_1$.
For each $y_i \in xs_1$, there exists a unique $z_i \in xs_2 \cap y_i s_1$, where $s_2$ is a unique element of $S \setminus T$ such that $c_{s_1 s_1}^{s_2} = 1$.
Note that $\{ z_i \mid 1 \leq i \leq p \}$ are distinct and $(z_i, z_{i+1}) \in t_1'$.

On the other hand, it is possible to take $\iota \in \mathrm{Aut}(T)$ such that $(t_1')^{\iota} = t_1$ and $\iota|_{H}$ is the identity map.
By Theorem \ref{thm:main2}, $\iota$ is extended to an element of $\mathrm{Aut}(S)$.
Let $(X, S')$ be the scheme obtained by applying Theorem \ref{thm:main1} for the case where $\iota_{_j} = \iota$ for the coset $X_j$ of $T$ containing $z_1$ and $\iota_{_i}$ is the identity except for $j$.

Next, we claim that $(X, S')$ is not isomorphic to $(X, S)$.
Suppose to the contrary that $\phi : X\cup S' \rightarrow X\cup S$ is an isomorphism such that for each $x, y \in X$,
\[ (x, y) \in s \Longleftrightarrow (x^\phi, y^\phi) \in s^\phi . \]
Then the adjacency among elements in $\{ x, y_i, z_i \mid 1\leq i \leq p \}$ must be invariant under $\phi$.
Note that both $(y_1, y_2)$ and $(z_1, z_2)$ lie in $t_1'$.
Therefore, both $(y_1^\phi, y_2^\phi)$ and $(z_1^\phi, z_2^\phi)$ lie in $t_1'^\phi$. However, such a structure does not occur in $(X, S)$ by the construction of $(X,S')$, a contradiction.

Finally, we claim that $(X, S')$ is not Schurian.
Let $x \in X$, $s_1 \in S \setminus T$ and $y_1 \in xs_1$. Then
a pair $(x, y_1)$ induces unique $y_2 \in xs_1 \cap y_1 t_1$ and $z_i \in xs_2 \cap y_i s_1$ for $i=1, 2$.
If $(X, S')$ is Schurian, then the relation containing $(z_1, z_2)$ is uniquely determined. On the other hand, by the construction of $(X, S')$, we can take $(x, y_1), (\hat{x}, \hat{y}_1) \in s$ such that the induced pairs $(z_1, z_2)$ and $(\hat{z}_1, \hat{z}_2)$ by $(x, y_1)$ and $(\hat{x}, \hat{y}_1)$ respectively, are not in the same relation. But, this is a contradiction.
\end{proof}

\section*{References}
\bibstyle{plain}

\end{document}